\documentclass[12pt,reqno]{amsart}

\usepackage{mathrsfs}
\usepackage{times}
\usepackage{graphicx}
\usepackage{hyperref}
\usepackage{amsmath,amsthm,amscd,amssymb,amsfonts}
\usepackage[all]{xy}

\theoremstyle{definition}
\newtheorem{theorem}{Theorem}[section]
\newtheorem{corollary}[theorem]{Corollary}
\newtheorem{definition}[theorem]{Definition}
\newtheorem{lemma}[theorem]{Lemma}

\newtheorem{proposition}[theorem]{Proposition}

\numberwithin{theorem}{section}
\numberwithin{equation}{section}

\newcommand{\norm}[1]{\left\Vert#1\right\Vert}

\newcommand{\la}{\langle}
\newcommand{\ra}{\rangle}
\newcommand{\B}{\mathcal{B}}
\newcommand{\Comp}{\mathbb{C}}
\newcommand{\g}{\mathbb{G}}
\newcommand{\n}{\mathbb{N}}
\newcommand{\V}{\mathcal{V}}

\newcommand{\bbGamma}{{\mathpalette\makebbGamma\relax}}
\newcommand{\makebbGamma}[2]{%
  \raisebox{\depth}{\scalebox{1}[-1]{$\mathsurround=0pt#1\mathbb{L}$}}%
}

\global\long\def\tp{\mathop{\xymatrix{*+<.7ex>[o][F-]{\scriptstyle \top}}
 } }

\usepackage[foot]{amsaddr}
\usepackage{float}

\begin{document}

\title[A characterization of freeness for finitely generated DQGs]{An analytic characterization of freeness for finitely generated discrete quantum groups}

\author[Yoonje Jeong]{Yoonje Jeong}
\address{\footnotesize Yoonje Jeong, Department of Mathematical Sciences, Seoul National University, Gwanak-Ro 1, Gwanak-Gu, Seoul 08826, Republic of Korea,
{\it jyj0601jyj@snu.ac.kr}}

\author[S.-G. Youn]{Sang-Gyun Youn}
\address{\footnotesize Sang-Gyun Youn, Department of Mathematics Education, Seoul National University, Gwanak-Ro 1, Gwanak-Gu, Seoul 08826, Republic of Korea,
{\it s.youn@snu.ac.kr}}

\begin{abstract}
    We prove that a freer quantum group has smaller moments of the self-adjoint main character in the category of finitely generated discrete quantum groups. As a result, the moments are minimized precisely by the unitary free quantum groups $\mathbb{F}U(Q)$. Furthermore, in the spirit of \cite{CaCo22}, we prove that the operator norm of the self-adjoint main character is minimized only by unitary free quantum groups, at least in the subcategory of duals of free quantum groups of Kac type.
\end{abstract}

\maketitle

\section{Introduction}

{\it Freeness} is a fundamental concept explored across various fields of mathematics for a long time. In the context of the group theory, {\it free groups} serve as the universal objects in the category of finitely generated discrete groups in the sense that any discrete group $\Gamma$ with $n$ generators is isomorphic to a quotient of the free group $\mathbb{F}_n$. In this case, the $n$ generators of $\mathbb{F}_n$ are subject to minimal constraints, ensuring their universality.

From an operator-algebraic perspective, discrete groups form a subclass of {\it locally compact quantum groups}, where {\it the Pontryagin duality} holds \cite{KuVa00,KuVa03}. For each discrete quantum group $\bbGamma$, there exists the dual compact quantum group $\g=\widehat{\bbGamma}$ such that $\widehat{\g}=\widehat{\widehat{\bbGamma}}=\bbGamma$. In particular, each discrete group $\Gamma$ admits a dual compact quantum group $\widehat{\Gamma}$ whose polynomial algebra $\text{Pol}(\widehat{\Gamma})$ is generated by the unitary operators $\lambda(g)\in \mathcal{U}(\ell^2(\Gamma))$ given by
\begin{equation}\label{eq10}
    (\lambda(g)f)(x)=f(g^{-1}x)
\end{equation}
for all $f\in \ell^2(\Gamma)$ and $x\in \Gamma$.

The concept of `freest' in the category of finitely generated discrete groups is analytically characterized by the operator norm. For finitely many generators $S=\left\{g_i\right\}_{i=1}^n$ of $\Gamma$, let us write $\chi_{\Gamma,S}=\sum_{i=1}^n \lambda(g_i)$ and $n(\Gamma,S)=\norm{\chi_{\Gamma,S}+\chi_{\Gamma,S}^*}_{B(\ell^2(\Gamma))}$. Then, the following inequality
\begin{equation}\label{eq00}
    2\sqrt{2n-1} \leq n(\Gamma,S)\leq 2n
\end{equation}
holds in general and, if $n\geq 2$, the minimum $2\sqrt{2n-1}$ is attained precisely when the generators $\left\{g_i\right\}_{i=1}^n$ freely generate $\Gamma$, i.e. 
\begin{equation}
    \Gamma=\la g_1\ra * \la g_2\ra * \cdots *\la g_n\ra  \cong \mathbb{F}_n .   
\end{equation}
In other words, the minimal constraints on $S=\left\{g_i\right\}_{i=1}^n$ force the operator norm $n(\Gamma,S)$ to be minimized. This fact comes from a well-known result of Kesten \cite{Kes59} and has been extended to the more general context of {\it free probability} in a recent work \cite{CaCo22}. 

On the other hand, the operator norm is not sufficient to characterize the concept of `freer' in the category of discrete groups. For example, let $\Gamma_2$ be an amenable discrete group with generators $S=\left\{g_i\right\}_{i=1}^n$ and let $\Gamma_1$ be an arbitrary quotient of $\Gamma_2$. Then, by \cite{Kes59}, we have
\begin{equation}
    n(\Gamma_2,S)=2n=n(\Gamma_1,\overline{S})
\end{equation}
where $\overline{S}=\left\{\overline{g_i}\right \}_{i=1}^n$ and $\overline{g_i}$ is the quotient image of $g_i$ in $\Gamma_1$.

In this paper, we employ a moment method to characterize both `freer' and `freest' within a broader category of {\it finitely generated discrete quantum groups}. Any finitely generated discrete quantum group $\bbGamma$ admits the dual compact matrix quantum group $\g$, and the polynomial algebra $\text{Pol}(\g)$ has finitely many generators $u_{ij}$ ($1\leq i,j\leq n$) such that $u=(u_{ij})_{i,j=1}^n$ is a unitary in $M_n(\text{Pol}(\g))$. We will focus on the following $k$-th moments
\begin{equation}\label{eq11}
    m_k(\bbGamma,u)=  h \left ( \left (\chi_{\bbGamma,u}+\chi_{\bbGamma,u}^* \right )^k \right ) ,
\end{equation}
where $\chi_{\bbGamma,u}=\sum_{i=1}^n u_{ii}$ is the {\it main character} and $h$ is the {\it Haar state} on $\g$. See Section \ref{sec-preliminaries} for more details. The freest object in this category is, namely, the {\it unitary free quantum groups} $\mathbb{F}U(Q)$ in the sense that any finitely generated discrete quantum groups $\bbGamma$ is a quotient of a unitary free quantum group $\mathbb{F}U(Q)$ for some invertible matrix $Q$.

In Section \ref{sec-moments}, we prove that a freer finitely generated discrete quantum group $(\bbGamma,u)$ has smaller moments $\left (m_k(\bbGamma,u)\right )_{k\in \n}$ (Theorem \ref{minimum}). Specifically, we prove that
\begin{equation}\label{ineq01}
    m_k(\bbGamma_2,v)\leq m_k(\bbGamma_1,u)
\end{equation}
for all natural numbers $k$, whenever $(\bbGamma_1,u)$ is a quotient of $(\bbGamma_2,v)$. Furthermore, we prove that equalities hold for all $k\in \n$ if and only if  $(\bbGamma_1,u)$ and $(\bbGamma_2,v)$ are isomorphic. Hence, we can conclude that the unitary free quantum groups $\bbGamma=\mathbb{F}U(Q)$ are indeed the only minimizers of the moments $\left ( m_k(\bbGamma,u)\right )_{k\in \n}$ (Corollary \ref{cor-main}).

In Section \ref{free section}, in the spirit of \cite{CaCo22}, we turn our attention back to the operator norm 
\begin{equation}
    n(\bbGamma,u)=\lim_{k\rightarrow \infty}m_{2k}(\bbGamma,u)^{\frac{1}{2k}}
\end{equation}
to investigate its connection with the notion of `freest' in the category of finitely generated discrete quantum groups. An immediate consequence of Corollary \ref{cor-main} is
\begin{equation}
    2\sqrt{2}= n(\mathbb{F}U(Q))\leq n(\bbGamma,u)
\end{equation}
for any finitely generated discrete quantum groups $(\bbGamma,u)$. Then a natural question is whether the minimum value $2\sqrt{2}$ is attained only when $(\bbGamma,u)$ is isomorphic to a unitary free quantum group. We provide a partial affirmative answer to this question, under the assumption that $\widehat{\bbGamma}$ is a free quantum group of Kac type, i.e. $S_n^+\subseteq \widehat{\bbGamma}\subseteq U_n^+$ (Theorem \ref{thm-free}).

\section{Preliminaries}\label{sec-preliminaries}

\subsection{Compact matrix quantum groups}

Let us begin with the definition of a Hopf $*$-algebra.

\begin{definition} \label{WoA}
    A Hopf $*$-algebra is given by $(A,\Delta,\epsilon,\kappa)$ where
    \begin{enumerate}
        \item $A$ is a unital $*$-algebra,
        \item $\Delta:A\rightarrow A\otimes A$ is a unital $*$-homomorphism such that 
           \begin{equation}
            (\Delta\otimes \text{id})\Delta = (\text{id}\otimes \Delta )\Delta,
            \end{equation}
        \item $\epsilon:A\rightarrow \Comp$ is a linear functional satisfying
        \begin{equation}
            (\epsilon\otimes \text{id})\Delta= \epsilon(\cdot)1 = (\text{id}\otimes \epsilon)\Delta,
        \end{equation}
        \item $\kappa:A\rightarrow A$ is a linear map satisfying
        \begin{equation}
                m(\text{id} \otimes \kappa)\Delta= \epsilon(\cdot)1 = m(\kappa \otimes \text{id})\Delta,
        \end{equation}
        where $m:A\otimes A\rightarrow A$ is the multiplication map.
    \end{enumerate}

    In addition, a Hopf $*$-algebra is called a {\it (algebraic) compact quantum group} if there exists a unital positive linear functional $h:A\rightarrow \Comp$ such that
    \begin{equation} \label{Haar}
        (h\otimes \text{id})\Delta=h(\cdot) 1 = (\text{id}\otimes h)\Delta.
    \end{equation}
    In this case, we denote by $\g=(A,\Delta,\epsilon,\kappa)$ and by $A=\text{Pol}(\g)$. Here, $\epsilon$ is a unital $*$-homomorphism called the {\it co-unit}, $\kappa$ is an anti-multiplicative map called the {\it antipode}, and $h$ is a faithful state called the {\it Haar state}.
\end{definition}

An element $v=(v_{ij})_{i,j=1}^{n_v}\in M_{n_v}(\text{Pol}(\g))$ is called a (finite dimensional unitary) {\it representation} if 
\begin{equation}
    \Delta(v_{ij})=\sum_{k=1}^{n_v}v_{ik}\otimes v_{kj}
\end{equation}
for all $1\leq i,j\leq n_v$. For representations $v_1\in M_{n_{v_1}}(\text{Pol}(\g))$ and $v_2\in M_{n_{v_2}}(\text{Pol}(\g))$, an operator $T\in B(\Comp^{n_{v_1}},\Comp^{n_{v_2}})$ is called an {\it intertwiner} if 
\begin{equation}
    (T\otimes 1)v_1 = v_2(T\otimes 1).
\end{equation}
In this paper, we will focus on the cases where we have a representation $u=(u_{ij})_{i,j=1}^n\in M_n(\text{Pol}(\g))$ such that 
\begin{enumerate}
    \item $\text{Pol}(\g)$ is generated by $\left\{u_{ij}\right\}_{i,j=1}^n$,
    \item $u^c=(u_{ij}^*)_{i,j=1}^n$ is invertible in $M_n(\text{Pol}(\g))$.
\end{enumerate}
In this case, $\g$ is called a {\it compact matrix quantum group} and we write $\g=(A,u)$. Here, $u$ is called the {\it fundamental representation} of $\g$, and $\chi_u=\sum_{i=1}^n u_{ii}$ is called the {\it main character}.

For two representations $v=\sum_{i,j=1}^m e_{ij}\otimes v_{ij}\in M_m(\text{Pol}(\g))$ and $w=\sum_{k,l=1}^n e_{kl}\otimes w_{kl}\in M_n(\text{Pol}(\g))$, the tensor product is defined as
\begin{equation}
    v\tp w=\sum_{i,j=1}^m \sum_{k,l=1}^n e_{ij}\otimes e_{kl}\otimes ( v_{ij}w_{kl} ),
\end{equation}
and the space of morphisms is defined as
\begin{equation}
    \text{Mor}(v,w)=\left\{T\in M_{n,m}{\Comp} : (T\otimes 1)v = w (T\otimes 1)\right\}.
    \end{equation}
For each representation $v\in M_n(\text{Pol}(\g))$, the dual is given as the contragredient representation $v^c=(v_{ij}^*)_{1\leq i,j\leq n}$. In addition, we have a natural linear isomorphism $\text{Mor} (v,w) \cong \text{Mor} (1, v^c \tp w)$.

\subsection{Finitely generated discrete quantum groups}

Recall that the Pontryagin duality implies a one-to-one correspondence between compact quantum groups and discrete quantum groups within the framework of locally compact quantum groups. The dual discrete quantum groups $\bbGamma$ of compact matrix quantum groups are called {\it finitely generated}. In addition, let us say that $\bbGamma$ \textit{has $n^2$ generators} when the fundamental representation of $\g$ is given by $u \in M_n(\text{Pol} (\mathbb{G}))$.

Let $(\bbGamma_1,u)$ and $(\bbGamma_2,v)$ be finitely generated discrete quantum groups with fundamental representations $u\in M_n(\text{Pol}(\g_1))$ and $v\in M_n(\text{Pol}(\g_2))$. We say that $\g_1$ is a quantum subgroup of $\g_2$, i.e. $\g_1\subseteq \g_2$ if we have a $*$-homomorphism $\phi:\text{Pol}(\g_2)\rightarrow \text{Pol}(\g_1)$ satisfying 
\begin{equation} \label{sub}
    (\text{id}_n\otimes \phi)(v)=u.    
\end{equation}
In this case, we say that {\it $\bbGamma_1$ is a quotient of $\bbGamma_2$} and write $\bbGamma_2\twoheadrightarrow \bbGamma_1$ at the level of finitely generated discrete quantum groups. From this point of view, it is reasonable to say that $\bbGamma_2$ is `\textit{freer}' than $\bbGamma_1$.

Recall that any finitely generated discrete quantum group $\bbGamma$ with $n^2$ generators allows an invertible matrix $Q\in M_n(\Comp)$ such that
\begin{equation}\label{eq21}
    (Q\otimes 1)u^c (Q^{-1}\otimes 1)
\end{equation}
is unitary, where $u=(u_{ij})_{1\leq i,j\leq n}$ is the fundamental representation of the dual compact matrix quantum group $\g$ \cite{Wo87b}. Moreover, in the category of finitely generated discrete quantum groups $\bbGamma$ with the additional condition that $(Q\otimes 1)u^c (Q^{-1}\otimes 1)$ is unitary, the universal object is the {\it unitary free quantum group} $\mathbb{F}U(Q)$. In this case, the dual of $\mathbb{F}U(Q)$ is denoted by $U_Q^+$, and the polynomial algebra $\text{Pol}(U_Q^+)$ is the universal unital $*$-algebra with a fundamental representation $u=(u_{ij})_{1\leq i,j\leq n}$ such that $(Q\otimes 1)u^c (Q^{-1}\otimes 1)$ is a unitary. We write $U_{\text{Id}_n}^+$ and $\mathbb{F}U(\text{Id}_n)$ simply as $U_n^+$ and $\mathbb{F}U_n$.

\subsection{The category of tensor representations}

Let $\bbGamma$ be a finitely generated discrete quantum group whose dual is a compact matrix quantum group $\g$ with a fundamental representation $u=(u_{ij})_{1\leq i,j\leq n}\in M_n(\text{Pol}(\g))$. Let us denote by $H=\Comp^n$ the fundamental representation space.

For any $\epsilon:[k]\rightarrow \left\{1,c\right\}$ with $[k]=\left\{1,2,\cdots,k\right\}$, let us denote by
\begin{align}
    &u^{\epsilon}= u^{\epsilon(1)}\tp u^{\epsilon(2)}\tp \cdots \tp u^{\epsilon(k)},\\
    & H^{\epsilon}=H^{\epsilon(1)}\otimes H^{\epsilon(2)}\otimes \cdots \otimes H^{\epsilon(k)}
\end{align}
where $H^c=\overline{H}\cong \Comp^n$. For any $i,j:[k]\rightarrow [n]$ let us write
\begin{equation}
    u^{\epsilon}_{ij}=u^{\epsilon(1)}_{i(1)j(1)}u^{\epsilon(2)}_{i(2)j(2)}\cdots u^{\epsilon(k)}_{i(k)j(k)}.
\end{equation}

For $\epsilon_1: [k_1] \rightarrow \{1,c\}$ and $\epsilon_2: [k_2] \rightarrow \{1,c\}$, we define their sum $\epsilon_1 + \epsilon_2 : [k_1+k_2]\rightarrow\{1,c\}$ by
\begin{equation}
    (\epsilon_1+\epsilon_2)(l) =
    \begin{cases}
        \epsilon_1(l) \quad &(1 \leq l \leq k_1) \\
        \epsilon_2(l-k_1) \quad &(k_1+1 \leq l \leq k_1+k_2)
    \end{cases}
    .
\end{equation}

For a finitely generated discrete quantum group, we consider a category $C^{\bbGamma}$ whose objects are $\left\{u^{\epsilon}\right\}_{\epsilon:[k]\rightarrow \left\{1,c\right\}}$ and the space of morphisms are given by $C^{\bbGamma}(\epsilon,\sigma)=\text{Mor}(u^{\epsilon},u^{\sigma})$ with $\epsilon:[k]\rightarrow \left\{1,c\right\}$ and $\sigma:[l]\rightarrow \left\{1,c\right\}$. In particular, for any $\epsilon:[k]\rightarrow \{1,c\}$, let us identify $C^{\bbGamma}(0,\epsilon)\subset B(\Comp, H^\epsilon)$ as a subspace of $H^{\epsilon}\cong H^{\otimes k}$ canonically, and take a basis $D^\bbGamma(\epsilon)$ of $C^\bbGamma(0,\epsilon)$. Then any element $p$ of the basis $D^{\bbGamma}(\epsilon)$ can be written as
\begin{equation}
    p=\sum_{i:[k]\rightarrow [n]} \langle p , e_i \rangle e_i
\end{equation}
where $e_i=e_{i(1)}\otimes e_{i(2)}\otimes \cdots \otimes e_{i(k)}$.

For any $\epsilon : [k] \rightarrow \{1,c\}$, let $NC(\epsilon)$ be the set of all non-crossing partitions on $[k]$. Note that this set depends only on $k$ and any partition $p$ allows a decomposition into disjoint blocks $p=\left\{V_1,V_2,\cdots,V_r\right\}$. Amongst important subsets of $NC(\epsilon)$ are $NC_2(\epsilon)$ and $\mathcal{NC}_2(\epsilon)$. Here, $NC_2(\epsilon)$ is the set of all non-crossing parings and $\mathcal{NC}_2(\epsilon)$ is the set of all $p \in NC_2(\epsilon)$ on $[k]$ satisfying the following additional condition:
\begin{equation}
    V \in p \Rightarrow \epsilon(V) = \{1,c\}.
\end{equation}

For $H=\mathbb{C}^n$ and a partition $p$ on $[k]$, we define $\xi_p \in H^\epsilon$ by
\begin{equation}
    \xi_p = \sum_{i:[k]\rightarrow [n]} \delta_p(i)  e_{i(1)} \otimes \cdots \otimes e_{i(k)}.
\end{equation}
Here $\delta_p(i) \in \{0,1\}$ is $1$ if and only if $p$ connects the same numbers when $i$ is written on $\epsilon$, i.e.
\begin{equation}
    V \in p \Rightarrow |i(V)| = 1.
\end{equation}

Note that, if $n\geq 2$, then $\{ \xi_p : p \in \mathcal{NC}_2(\epsilon)\}$ and $\{\xi_p: p \in NC_2(\epsilon)\}$ are bases of $C^{\mathbb{F}U_n}(0, \epsilon)$ and $C^{\mathbb{F}O_n}(0,\epsilon)$, respectively. Here, $\mathbb{F}O_n$ is the dual of the {\it free orthogonal quantum group} $O_n^+$ \cite{Wa95}. Furthermore, if $n\geq 4$, then $\{\xi_p: p \in NC(\epsilon)\}$ is a basis of $C^{\mathbb{F}S_n}(0,\epsilon)$ where $\mathbb{F}S_n$ is the dual of the {\it free permutation quantum group} $S_n^+$ \cite{Wa98}. See \cite{BaCoper} and \cite{BaCo07} for more details. A \textit{free quantum group} (of Kac type) is a compact matrix quantum group $\mathbb{G}$ with $S_n^+ \subseteq \mathbb{G} \subseteq U_n^+$.

\section{A moment method to characterize freeness}\label{sec-moments}

Let $\bbGamma$ be a finitely generated discrete quantum group and let $\g$ be the associated dual compact matrix quantum group with a fundamental unitary representation $u=(u_{ij})_{1\leq i,j\leq n}$. We call
\begin{equation}
    \chi_{\bbGamma,u}=\sum_{i=1}^n u_{ii}
\end{equation}
the {\it main character}. Let us denote the $k$-th moments of the \textit{self-adjoint main character} $\chi_{\bbGamma}+\chi_{\bbGamma}^*$ by
\begin{equation}
    m_k(\bbGamma,u)=h\left ( \left ( \chi_{\bbGamma,u}+\chi_{\bbGamma,u}^* \right )^k\right )
\end{equation}
with respect to the Haar state on $\g$. We often simply write $m_k(\bbGamma)$ if there is no possibility of confusion. In particular, if $\bbGamma$ is a discrete group $\Gamma$ with generators $S=(g_i)_{i=1}^n$, then the fundamental unitary representation is given by $u=\sum_{i=1}^n e_{ii}\otimes \lambda(g_i)$ and the main character is
\begin{equation}
    \chi_{\Gamma,u}= \chi_{\Gamma,S}=\sum_{i=1}^n \lambda(g_i).
\end{equation}
Then, it is straightforward to see that 
\begin{align}  \label{relation}
    &m_k(\Gamma,u) = m_k(\Gamma,S)\\
    &=\left | \left\{(x_1,x_2,\cdots,x_k)\in (S')^k: x_1x_2\cdots x_k=e\right\} \right |
\end{align}
where $S'=(g_i^{\epsilon})_{1\leq i\leq n, \epsilon\in \left\{\pm 1 \right\}}$. This leads us to the following observations in the category of finitely generated discrete groups.

\begin{proposition} \label{group case}
    Let $\Gamma$ be a finitely generated discrete group with generators $S=(g_i)_{i=1}^n$, and let $R$ be a normal subgroup of $\Gamma$. Then we have
    \begin{equation} \label{ineq}
        m_k(\Gamma,S) \leq m_k(\Gamma/R,\overline{S})
    \end{equation}
    for all natural numbers $k$, where the quotient group $\Gamma/R$ is generated by $\overline{S}=(g_iS)_{i=1}^n$. In addition, the normal subgroup $R$ is non-trivial if and only if $m_k(\Gamma)<m_k(\Gamma/R)$ for some natural number $k$.
\end{proposition}

\begin{proof}
    Let us denote by 
    \begin{equation}
        E_k(\Gamma,S)=\left\{(x_1,x_2,\cdots,x_k)\in (S')^k: x_1x_2\cdots x_k=e\right\}
    \end{equation}
    where $S'=(g_i^{\epsilon})_{1\leq i\leq n, \epsilon\in \left\{\pm 1 \right\}}$, and by $q:\Gamma\rightarrow \Gamma/R$ the quotient map. Since $(x_1,x_2,\cdots,x_k)\in E_k(\Gamma,S)$ implies
    \begin{equation}
        (q(x_1),q(x_2),\cdots ,q(x_k))\in E_k(\Gamma/R,q(S)),
    \end{equation}
    we obtain \eqref{ineq} for all $k\in \n$. Note that $R$ is non-trivial if and only if there exists a non-trivial element $x\in R$. Let us write $x=x_1x_2\cdots x_k$ with $x_i\in S'$ for all $1\leq i\leq k$. Then $(x_1,x_2,\cdots,x_k)\notin E_k(\Gamma,S)$, but we have $(q(x_1),q(x_2),\cdots,q(x_k))\in E_k(\Gamma/R,q(S))$ since
    \begin{equation}
        q(x_1)q(x_2)\cdots q(x_k)=q(x_1x_2\cdots x_k)=q(x)=R.
    \end{equation}
\end{proof}

\begin{corollary}
    For any discrete group $\Gamma$ with $n$ generators $S=(g_i)_{i=1}^n$, we have
    \begin{equation}
        m_k(\mathbb{F}_n) \leq m_k(\Gamma,S)
    \end{equation}
    for all natural numbers $k$. Furthermore, $m_k(\mathbb{F}_n)=m_k(\Gamma,S)$ holds for all natural numbers $k$ if and only if $\Gamma$ is isomorphic to the free group $\mathbb{F}_n$.
\end{corollary}

One of the main purposes of this paper is to extend Proposition \ref{group case} for finitely generated discrete groups $(\Gamma,S)$ to a broader category of finitely generated discrete quantum groups $(\bbGamma,u)$. First of all, the following proposition tells us that the $k$-th moments $m_k(\bbGamma,u)$ are determined by the dimensions of the space of morphisms $C^{\bbGamma}(0,\epsilon)$.

\begin{proposition} \label{moments}
    Let $(\bbGamma,u)$ be a finitely generated discrete quantum group. Then we have
        \begin{equation} \label{goal moment}
            m_k(\bbGamma,u) = \sum_{\epsilon : [k] \rightarrow \{1,c\}} \dim C^\bbGamma(0,\epsilon)
        \end{equation}
        for all natural numbers $k$.
\end{proposition}

\begin{proof}
    Note that the $k$-th moment $m_k(\bbGamma,u)$ can be written as
    \begin{align}
       h\left ( (\chi_{\bbGamma}+\chi_{\bbGamma}^*)^k\right ) =\sum_{\epsilon:[k]\rightarrow \{1,c\}} h\left ( \chi_{\bbGamma}^{\epsilon(1)}\chi_{\bbGamma}^{\epsilon(2)}\cdots \chi_{\bbGamma}^{\epsilon(k)}\right )
    \end{align}
    and $h\left ( \chi_{\bbGamma}^{\epsilon(1)}\chi_{\bbGamma}^{\epsilon(2)}\cdots \chi_{\bbGamma}^{\epsilon(k)}\right )$ counts the number of trivial representations that appear in the irreducible decomposition of the representation 
    \begin{equation}
        u^{\epsilon}=u^{\epsilon(1)}\tp u^{\epsilon(2)}\tp \cdots \tp u^{\epsilon(k)},
    \end{equation}
    so we can conclude that 
    \begin{align}
       h\left ( (\chi_{\bbGamma}+\chi_{\bbGamma}^*)^k\right ) =\sum_{\epsilon:[k]\rightarrow \{1,c\}} \dim C^{\bbGamma}(0,\epsilon).
    \end{align}
\end{proof}

Applying the categorical description of $m_k(\bbGamma,u)$ in Proposition \ref{moments}, we can now extend Proposition \ref{group case} to the case of general finitely generated discrete quantum groups.

\begin{theorem} \label{minimum}
    Let $(\bbGamma_1,u)$ and $(\bbGamma_2,v)$ be finitely generated discrete quantum groups. If $(\bbGamma_1,u)$ is a quotient of $(\bbGamma_2,v)$, i.e. $\bbGamma_2 \twoheadrightarrow \bbGamma_1$, then we have
    \begin{equation} \label{ineq quantum}
        m_k(\bbGamma_2,v) \leq m_k(\bbGamma_1,u)
    \end{equation}
    for all natural numbers $k$. In addition, if $m_k(\bbGamma_2,v)=m_k(\bbGamma_1,u)$ for all $k\in \n$, then $\bbGamma_1$ and $\bbGamma_2$ are isomorphic.
\end{theorem}
\begin{proof}   
    Recall that we have
    \begin{align}
        m_k(\bbGamma_1,u) &= \sum_{\epsilon : [k] \rightarrow \{1,c\}} \dim C^{\bbGamma_1}(0,\epsilon)\\
        m_k(\bbGamma_2,v) &= \sum_{\epsilon : [k] \rightarrow \{1,c\}} \dim C^{\bbGamma_2}(0,\epsilon)
    \end{align}
    by Proposition \ref{moments}. Since $\dim C^{\bbGamma_1}(0,\epsilon)$ (resp. $\dim C^{\bbGamma_2}(0,\epsilon)$) counts the number of the trivial representation in an irreducible decomposition of the representation $v^{\epsilon}=v^{\epsilon(1)}\tp \cdots \tp v^{\epsilon(k)}$ (resp. $u^{\epsilon}=u^{\epsilon(1)}\tp \cdots \tp u^{\epsilon(k)}$), it is enough to prove that $v^{\epsilon}$ contains more trivial representations than $u^{\epsilon}$. Let us suppose that there exist non-trivial mutually inequivalent irreducible representations $u^{\gamma_1}$, $u^{\gamma_2}$, $\cdots$, $u^{\gamma_r}$ with multiplicities $n_1,n_2,\cdots,n_r\in \n$ such that
    \begin{equation}\label{eq30}
        u^{\epsilon}=n_1 u^{\gamma_1}\oplus n_2 u^{\gamma_2}\oplus \cdots \oplus n_r u^{\gamma_r}\oplus n_{r+1}1_{\g_2}
    \end{equation}
    where $1_{\g_2}$ is the trivial representation of $\g_2$ with the multiplicity $n_{r+1}\in \n_0$. Recall that $\bbGamma_2\twoheadrightarrow\bbGamma_1$ implies that there is a $*$-homomorphism $\phi:\text{Pol}(\g_2)\rightarrow \text{Pol}(\g_1)$ satisfying
    \begin{equation}
        (\text{id}_n\otimes \phi)(u)=v,
    \end{equation}
    implying $(\text{id}_{n^k}\otimes \phi)(u^{\epsilon})=v^{\epsilon}$ for any $\epsilon:[k]\rightarrow \{1,c\}$. Moreover, since 
    \begin{equation}\label{eq31}
        v^{\epsilon}= \left ( \bigoplus_{i=1}^r (\text{id}_{n_{i}}\otimes \phi)(u^{\gamma_i}) \right ) \oplus n_{r+1}1_{\g_1}
    \end{equation}
    and each $v_{i}=(\text{id}_{n_{i}}\otimes \phi)(u^{\gamma_i})$ is a representation of $\g_1$, the number of trivial representations in the irreducible decomposition of $v^{\epsilon}$ is greater than or equal to $n_{r+1}$, which is the number of trivial representation in $u^{\epsilon}$.
    
    Now, let us suppose that $\bbGamma_2\twoheadrightarrow \bbGamma_1$ is strict, i.e. there exists $\epsilon_0:[k]\rightarrow \{1,c\}$ and coefficients $(a_{i_0j_0})_{i_0,j_0\in [k]\rightarrow [n]}$ such that 
    \begin{equation}
        x_0=\sum_{i_0,j_0\in [k]\rightarrow [n]}a_{i_0j_0}u^{\epsilon_0}_{i_0j_0}\neq 0
    \end{equation}
    and $\phi(x_0)=\sum_{i_0,j_0\in [k]\rightarrow [n]}a_{i_0j_0}v^{\epsilon_0}_{i_0j_0}=0$. Then $x=x_0^*x_0\neq 0$ and $\phi(x)=\phi(x_0)^*\phi(x_0)=0$, and there exists $\epsilon:[2k]\rightarrow \{1,c\}$ and coefficients $(b_{ij})_{i,j:[2k]\rightarrow [n]}$ such that 
    \begin{equation}
        x=x_0^*x_0=\sum_{i,j\in [2k]\rightarrow [n]}b_{ij}u^{\epsilon}_{ij}.
    \end{equation}
    
    We may assume that $u^{\epsilon}$ is of the form \eqref{eq30} with $n_{r+1}>0$. Then $x'=x-n_{r+1}1_{\g_2}$ satisfies $h(x')=0$ and $\phi(x')=\phi(x)-n_{r+1}1_{\g_1}=-n_{r+1}1_{\g_1}$. This implies $h(\phi(x'))=-n_{r+1}\neq 0$, which is impossible if the trivial representation is not contained in any irreducible decompositions of all the representations $v_{i}=(\text{id}_{n_{i}}\otimes \phi)(u^{\gamma_i})$ of $\g_1$ in \eqref{eq31}. Thus, at least one $v_i$ contains the trivial representation, so the number of trivial representations of $v^{\epsilon}$ is greater than that of $u^{\epsilon}$. 
\end{proof}

As the unitary free quantum group $\mathbb{F}U(Q)$ is the universal object among all finitely generated discrete quantum groups $(\bbGamma,u)$ satisfying that $(Q\otimes 1)u^c (Q^{-1}\otimes 1)$ is a unitary, we obtain the following Corollary.

\begin{corollary}\label{cor-main}
    For any finitely generated discrete quantum group $(\bbGamma,u)$ with $n^2$ generators, we have
    \begin{equation} \label{eqthis}
        \begin{cases}
            m_k(\bbGamma,u) \geq 2^{\frac{k}{2}} C_{\frac{k}{2}} \quad &(\text{$k$ is even}) \\
            m_k(\bbGamma,u) \geq 0 \quad &(\text{$k$ is odd})
        \end{cases}
        .
    \end{equation}
    In addition, if equalities hold in \eqref{eqthis} for all $k \in \mathbb{N}$, then $\bbGamma$ is isomorphic to $\mathbb{F}U(Q)$ for some invertible matrix $Q \in M_n(\mathbb{C})$.
\end{corollary}

\begin{proof}
    Note that for any finitely generated discrete quantum group $(\bbGamma,u)$, there exists an invertible matrix $Q \in M_n(\mathbb{C})$ such that $\mathbb{F}U(Q) \twoheadrightarrow \bbGamma$. By Theorem \ref{minimum}, we have
    \begin{align}
        m_k(\bbGamma,u) &\geq m_k(\mathbb{F}U(Q))= m_k(\mathbb{F}U_n)=\sum_{\epsilon : [k]\rightarrow \{1,c\}} \dim C^{\mathbb{F}U_n}(0,\epsilon) \\
        &= \sum_{\epsilon : [k] \rightarrow \{1,c\}} \left|\mathcal{NC}_2(0,\epsilon)\right| = \sum_{\substack{\epsilon : [k] \rightarrow \{1,c\} \\ |\epsilon^{-1}(1)| = |\epsilon^{-1}(c)|}} |\mathcal{NC}_2(0,\epsilon)|
    \end{align}
    for all $k \in \mathbb{N}$. Here, the first equality is thanks to the fact that $\mathbb{F}U(Q)$ and $\mathbb{F}U_n$ share the same fusion rule. If $k$ is odd, there is no $\epsilon:[k]\rightarrow\{1,c\}$ satisfying $|\epsilon^{-1}(1)| = |\epsilon^{-1}(c)|$, hence $m_k(\mathbb{F}U_n) =0$. If $k$ is even, we have
    \begin{equation}
        m_k(\mathbb{F}U_n) = \sum_{\substack{\epsilon : [k] \rightarrow \{1,c\} \\ |\epsilon^{-1}(1)| = |\epsilon^{-1}(c)|}} |\mathcal{NC}_2(0,\epsilon)|
         = 2^{\frac{k}{2}} C_{\frac{k}{2}}.
    \end{equation}

    For the second assertion, suppose that equalities hold in \eqref{eqthis} for all $k \in \mathbb{N}$. As above, there exists an invertible matrix $Q \in M_n(\mathbb{C})$ such that $\mathbb{F}U(Q) \twoheadrightarrow \bbGamma$ and
    \begin{equation}
        m_k(\bbGamma)  = m_k(\mathbb{F}U(Q))
    \end{equation}
    for all $k \in \mathbb{N}$. Then we conclude that $\bbGamma$ and $\mathbb{F}U(Q)$ are isomorphic by Theorem \ref{minimum}.
\end{proof}

\section{Partial answers for the operator norm} \label{free section}

One might wonder whether a parallel result of Theorem \ref{minimum} holds particularly for the operator norm. As mentioned in the Introduction, the operator norm is not enough to characterize `freer', but can characterize `freest' within the category of finitely generated discrete groups. 

In this section, let us focus on the following (reduced) operator norm 
\begin{equation} \label{short}
    n(\bbGamma,u) =\lim_{k\rightarrow \infty}  m_{2k}(\bbGamma,u)^{\frac{1}{2k}} =\lim_{k\rightarrow \infty} \left[ h \left ( \left (\chi_{\bbGamma,u}+\chi_{\bbGamma,u}^*\right )^{2k} \right ) \right]^{\frac{1}{2k}}
\end{equation}
within the category of finitely generated discrete quantum groups. Note that, for any invertible matrix $Q\in M_n(\Comp)$, we observed
   \begin{equation}
            m_{2k}(\mathbb{F}U(Q))= 2^{k} C_{k}
        \end{equation}
for all natural numbers $k$ in the proof of Corollary \ref{cor-main}. This implies
\begin{equation}\label{eq48}
    n(\mathbb{F}U(Q))=\lim_{k\rightarrow \infty}\sqrt{2} \cdot C_k^{\frac{1}{2k}}=2\sqrt{2},
\end{equation}
and the following Proposition follows from the conclusions in Section 3.

\begin{proposition}\label{prop40}
    Let $(\bbGamma_1,u)$ and $(\bbGamma_2,v)$ be finitely generated discrete quantum groups with $n^2$ generators. If $(\bbGamma_1,u)$ is a quotient of $(\bbGamma_2,v)$, i.e. $\bbGamma_2\twoheadrightarrow \bbGamma_1$, then $n(\bbGamma_2,v) \leq n(\bbGamma_1,u)$. In particular, if $n\geq 2$, we have
    \begin{equation}\label{eq40}
        2\sqrt{2}\leq n(\bbGamma)
    \end{equation}
    for any finitely generated discrete quantum group $\bbGamma$ with $n^2$ generators.
\end{proposition}

\begin{proof}
    The first conclusion is a direct consequence from Theorem \ref{minimum} and \eqref{short} since $\bbGamma_2\twoheadrightarrow \bbGamma_1$ implies 
    \begin{equation} \label{norm monotone}
         n(\bbGamma_2,v)=\lim_{k\rightarrow \infty}m_{2k}(\bbGamma_2,v)^{\frac{1}{2k}}\leq \lim_{k\rightarrow \infty}m_{2k}(\bbGamma_1,u)^{\frac{1}{2k}} = n(\bbGamma_1,u).
    \end{equation}
    The second conclusion follows from \eqref{eq48}.
\end{proof}

A natural question on Proposition \ref{prop40} is whether the unitary free quantum groups $\mathbb{F}U(Q)$ are the only finitely generated discrete quantum groups $(\bbGamma,u)$ satisfying $n(\bbGamma,u)=2\sqrt{2}$. We have no complete answer for now but will present an affirmative answer under the assumption $\mathbb{F}U_n\twoheadrightarrow \bbGamma\twoheadrightarrow \mathbb{F}S_n$. To do this, our strategy is to focus on a sequence $(A_k)_{k=0}^{\infty}$ given by
\begin{equation} \label{sequence}
    A_k= \frac{1}{2^k} m_{2k}(\bbGamma),~k\in \n \cup \{0\}.
\end{equation}
A key step is to prove the following property of the above sequence $(A_k)_{k=0}^{\infty}$.

\begin{lemma}\label{lem-sequence}
    Suppose that a finitely generated discrete quantum group $\bbGamma$ with $n^2$ generators satisfies one of the following conditions:
    \begin{enumerate}
        \item $n\geq 4$ and $\mathbb{F}U_n\twoheadrightarrow \bbGamma\twoheadrightarrow \mathbb{F}S_n$.
        \item $n\geq 2$ and $\mathbb{F}U_n\twoheadrightarrow \bbGamma\twoheadrightarrow \mathbb{F}O_n$.
    \end{enumerate}
    Then the sequence $(A_k)_{k=0}^{\infty}$ in \eqref{sequence} satisfies the following properties:
    \begin{enumerate}
        \item[(a)] $A_0=1$.
        \item[(b)] $\forall k \in \n_0$, $A_{k+1} \geq \sum_{j=0}^k A_j A_{k-j}$.
    \end{enumerate}
    Furthermore, if $\bbGamma$ is not isomorphic to $\mathbb{F}U_n$, then we obtain the following additional property:
    \begin{enumerate}
        \item[(c)] $\exists k_0 \in \mathbb{N}$ s.t. $A_{k_0} > C_{k_0}= \frac{1}{k_0+1} \binom{2k_0}{k_0}$.
    \end{enumerate}
\end{lemma}

\begin{proof}
  Let us suppose that $\bbGamma$ satisfies the first condition (1). Then the property (a) is clear, so let us focus on the second property (b). Note that, for any $\epsilon : [2k+2] \rightarrow \{1,c\}$ and $1\leq r\leq k+1$, we have a unique decomposition
    \begin{equation}
        \epsilon = \epsilon(1) + \sigma + \epsilon(2r) + \sigma'
    \end{equation}
    with  $\sigma=\sigma(\epsilon,r) : [2r-2] \rightarrow \{1,c\}$ and $\sigma'=\sigma'(\epsilon,r) : [2k+2-2r] \rightarrow \{1,c\}$. An example for the case $(k,r)=(4,3)$ is described in the following figure:
    \begin{figure}[H]
        \includegraphics[width=12cm]{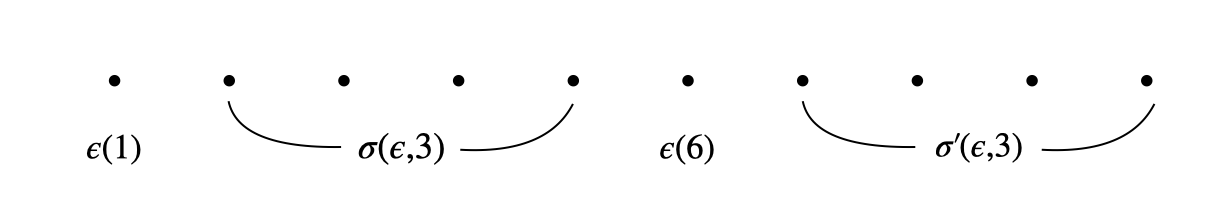}
    \end{figure}
    This allows us to obtain
    \begin{align}
        &2 \cdot  \sum_{r=0}^k m_{2r}(\bbGamma) \cdot m_{2k-2r}(\bbGamma) \\
        &= 2 \cdot \sum_{r=0}^k \sum_{\substack{\sigma:[2r]\rightarrow\{1,c\} \\ \sigma' : [2k-2r] \rightarrow \{ 1,c\} }} \dim C^\bbGamma(0,\sigma) \cdot \dim C^\bbGamma(0,\sigma')  \\
        &= \sum_{r=0}^k\sum_{\epsilon:[2k+2]\rightarrow\{1,c\}} \delta_{\epsilon(1) \epsilon(2r)} \cdot \dim C^\bbGamma(0,\sigma(\epsilon,r)) \cdot \dim C^\bbGamma(0,\sigma'(\epsilon,r))\\
        &= \sum_{\epsilon:[2k+2]\rightarrow \{1,c\}} \sum_{r=0}^k \delta_{\epsilon(1) \epsilon(2r)} \cdot \dim C^\bbGamma(0,\sigma(\epsilon,r)) \cdot \dim C^\bbGamma(0,\sigma'(\epsilon,r))\\
        &=\sum_{\epsilon:[2k+2]\rightarrow \{1,c\}} \sum_{r \in R(\epsilon)} \dim C^\bbGamma(0,\sigma(\epsilon,r)) \cdot \dim C^\bbGamma(0,\sigma'(\epsilon,r))
    \end{align}
    where $R(\epsilon) = \{ r \in [k+1] : \epsilon(1) \neq \epsilon(2r) \}$ for each $\epsilon:[2k+2]\rightarrow \{1,c\}$. Thus, it is enough to prove
    \begin{equation}
        \sum_{r \in R(\epsilon)} \dim C^\bbGamma(0,\sigma(\epsilon,r)) \cdot \dim C^\bbGamma(0,\sigma'(\epsilon,r))\leq \dim C^{\bbGamma}(0,\epsilon)
    \end{equation}
    to reach the desired conclusion. First, it is straightforward to see that
    \begin{equation}
        \xi_{v, v'}=\sum_{i=1}^n e_i \otimes v \otimes e_i \otimes v' \in C^{\bbGamma}(0,\epsilon)
    \end{equation}
    for any $v \in C^{\bbGamma}(0,\sigma(\epsilon,r))$ and $v' \in C^{\bbGamma}(0,\sigma'(\epsilon,r))$, and this implies
    \begin{equation}
       \sum_{r\in R}W_r \subseteq C^{\bbGamma}(0,\epsilon)
    \end{equation}
    where  $W_r$ is a subspace defined as
    \begin{equation}
        W_r=\text{span}\left\{ \xi_{v,v'} : v \in  C^{\bbGamma}(0,\sigma(\epsilon,r)), v' \in C^{\bbGamma}(0,\sigma'(\epsilon,r))\right \}.
    \end{equation}
    To prove the following fact
    \begin{align}\label{eq44}
       \sum_{r\in R}W_r=\bigoplus_{r\in R}W_r,
    \end{align}
    let us consider a larger subspace
    \begin{equation}
        V_r=\text{span} \left\{ \sum_{i=1}^n e_i \otimes \xi_p \otimes e_i \otimes \xi_{p'} : p \in NC(\sigma(\epsilon,r)), p' \in NC(\sigma'(\epsilon,r)) \right\}
    \end{equation}
    and let us prove a stronger property
    \begin{align}\label{eq43}
     \sum_{r\in R} V_r = \bigoplus_{r\in R}V_{r}.
    \end{align}
    
    A key idea for the proof is to note that, for any $p\in NC(\sigma(\epsilon,r))$ and $p'\in NC(\sigma'(\epsilon,r))$, there exists a unique $q=q(r,p,p')\in NC(\epsilon)$ such that
    \begin{equation}
        \xi_q=\sum_{i=1}^n e_i \otimes \xi_p \otimes e_i \otimes \xi_{p'}
    \end{equation}
    as described in the following figure.
    \begin{figure}[H]
        \includegraphics[width=12cm]{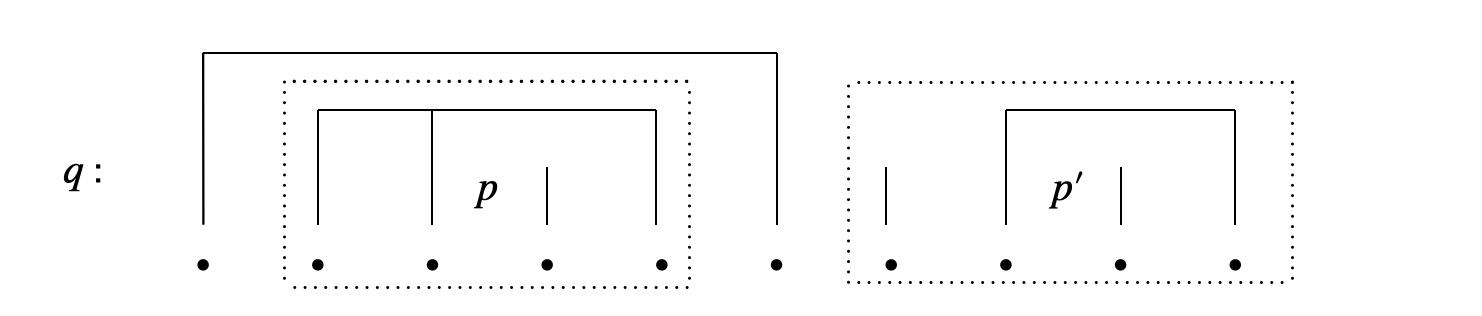}
    \end{figure}
    \noindent Furthermore, the mapping $(r,p,p')\mapsto q\in NC(\epsilon)$ is injective, so we can conclude that 
    \begin{equation}
        \B=\bigcup_{r\in R}\B_r=\bigcup_{r\in R} \left\{\xi_{q}\right\}_{p \in NC(\sigma(\epsilon,r)), p' \in NC(\sigma'(\epsilon,r))}
    \end{equation}
    forms a linearly independent subset thanks to the assumption $n\geq 4$. Now, $\B_r$'s are mutually disjoint subsets of a linearly independent set, so the generated subspaces $V_r$'s satisfy the desired property \eqref{eq43}, and consequently \eqref{eq44}. Finally, this leads us to the following inequality:
    \begin{align}
        \dim & C^\bbGamma(0,\epsilon)  \geq \dim \left ( \bigoplus_{r\in R}W_r \right )\\
        &= \sum_{r\in R} \dim \left ( C^{\bbGamma}(0,\sigma(\epsilon,r)) \otimes C^{\bbGamma}(0,\sigma'(\epsilon,r))\right )\\ & =\sum_{r\in R}\dim C^\bbGamma(0, \sigma(\epsilon,r)) \cdot \dim C^\bbGamma(0, \sigma'(\epsilon,r)).
    \end{align}
    
    From now, let us suppose that  $\bbGamma$ is not isomorphic to $\mathbb{F}U_n$ to prove (c). By Theorem \ref{minimum}, there exists $l_0\in \n_0$ such that 
    \begin{equation}
        m_{l_0}(\bbGamma) > m_{l_0}(\mathbb{F}U_n).
    \end{equation}
    If $l_0=2k_0$ is even, then we have
    \begin{align}
        2^{k_0} A_{k_0} = m_{l_0}(\bbGamma)  &>  m_{l_0}(\mathbb{F}U_n) = 2^{k_0} C_{k_0}.
    \end{align}
        
    Let us suppose that $l_0$ is odd. There exists $\epsilon : [l_0]\rightarrow \{1,c\}$ with $\dim C^{\bbGamma}(0,\epsilon) >0$, so we can take a non-zero element $v$ of $C^{\bbGamma}(0,\epsilon)$. For $p, q \in NC(\epsilon)$, we denote the juxtaposition of $p$ and $q$ by $[pq] \in NC(\epsilon + \epsilon)$. In particular, we have $\xi_{[pq]} = \xi_p \otimes \xi_q$. Now, let us consider the following two subspaces
    \begin{align} 
       \V_1 &= \text{span} \{ \xi_{[pq]} : p,q \in NC(\epsilon) \} =\text{span} \{ \xi_p \otimes \xi_q : p,q \in NC(\epsilon) \}\\
       \V_2 &= C^{\mathbb{F}U_n}(0,\epsilon + \epsilon) = \text{span} \{ \xi_p : p \in \mathcal{NC}_2(\epsilon+\epsilon) \},
    \end{align}
    and the given assumption $\bbGamma \twoheadrightarrow \mathbb{F}S_n$ implies the following:
    \begin{align}
       \label{eq46} v\otimes v&\in C^{\bbGamma}(0,\epsilon)\otimes C^{\bbGamma}(0,\epsilon)\subseteq C^{\bbGamma}(0,\epsilon+\epsilon)\cap \V_1,\\
        \V_1+\V_2&\subseteq \text{span}\left\{\xi_r: r\in NC(\epsilon+\epsilon)\right\}=C^{\mathbb{F}S_n}(0,\epsilon+\epsilon).
    \end{align}
    Here, $\left\{\xi_r: r\in NC(\epsilon+\epsilon)\right\}$ is a basis of $C^{\mathbb{F}S_n}(0,\epsilon+\epsilon)$ thanks to the assumption $n\geq 4$. Moreover, since $l_0$ is odd, it is immediate to see that
    \begin{equation}
        \{ [pq] : p,q \in NC(\epsilon)\} \cap \mathcal{NC}_2(\epsilon + \epsilon) = \emptyset.
    \end{equation}
    This implies that $\{ \xi_{[pq]} : p,q \in NC(\epsilon) \}$ and $\{ \xi_p : p \in \mathcal{NC}_2(\epsilon+\epsilon) \}$ are mutually disjoint subsets of a linearly independent set, and we obtain $\V_1+\V_2=\V_1\oplus \V_2$. Thus, combining with \eqref{eq46}, we can see that
    \begin{equation}
        v\otimes v \in C^{\bbGamma}(0,\epsilon+\epsilon)\setminus \V_2=C^{\bbGamma}(0,\epsilon+\epsilon)\setminus  C^{\mathbb{F}U_n}(0,\epsilon+\epsilon),
    \end{equation}
    implying the following strict inequality.
    \begin{equation}
        \dim C^{\bbGamma}(0,\epsilon+\epsilon) > \dim C^{\mathbb{F}U_n}(0,\epsilon+\epsilon),
    \end{equation}
    Consequently, we can conclude that
    \begin{equation}
        2^{l_0} A_{l_0} = m_{2l_0}(\bbGamma) > m_{2l_0}(\mathbb{F}U_n) = 2^{l_0} C_{l_0}.
    \end{equation}

    The proof for the condition (2) is similar. The only thing to note is that the roles of $NC(\epsilon)$ and $\mathbb{F}S_n$ are replaced by those of $NC_2(\epsilon)$ and $\mathbb{F}O_n$, respectively.
\end{proof}

The properties (a),(b),(c) in Lemma \ref{lem-sequence} allow us to establish the following conclusion.

\begin{theorem}\label{thm-free}
     Suppose that a finitely generated discrete quantum group $\bbGamma$ with $n^2$ generators satisfies one of the following conditions:
    \begin{enumerate}
        \item $n\geq 4$ and $\mathbb{F}U_n\twoheadrightarrow \bbGamma\twoheadrightarrow \mathbb{F}S_n$.
        \item $n\geq 2$ and $\mathbb{F}U_n\twoheadrightarrow \bbGamma\twoheadrightarrow \mathbb{F}O_n$.
    \end{enumerate}
    If $n(\bbGamma)=n(\mathbb{F}U_n)$, then $\bbGamma$ is isomorphic to $\mathbb{F}U_n$.
\end{theorem}

\begin{proof}
    Let $(A_k)_{k=1}^\infty$ be the sequence discussed in \eqref{sequence} and Lemma \ref{lem-sequence}. Note that it is enough to prove 
    \begin{equation}
        \limsup_{k\rightarrow \infty}A_k^{\frac{1}{k}}>4
    \end{equation}
    to reach the conclusion since $n(\mathbb{F}U_n)=2\sqrt{2}$ or, equivalently, the radius of convergence of the following power series
    \begin{equation}
        A(x) = \sum_{k \geq 0} A_k x^k
    \end{equation}
    is strictly smaller than $\frac{1}{4}$. To do this, let us assume that $A(x)$ converges for all $x\in \left(-\frac{1}{4},\frac{1}{4}\right)$. By Lemma \ref{lem-sequence}, there exists $l_0\in \n$ such that 
    \begin{equation}\label{eq47}
        \varepsilon = \frac{1}{8^{l_0+1}}\left(A_{l_0+1} - \sum_{j=0}^{l_0} A_j A_{l_0-j} \right)>0.
    \end{equation}
    If not, $A_k$ should coincide with $C_k$ for all $k$. Now, for any $x\in \left(\frac{1}{8},\frac{1}{4}\right)$, we have
    \begin{align}
        &A(x)-1 = \sum_{k=0}^{\infty}A_{k+1} x^{k+1} =\sum_{k\in \n_0 \setminus \{l_0\}}A_{k+1} x^{k+1}+ A_{l_0+1}x^{l_0+1}  \\
        &\geq \sum_{k\in \n_0 \setminus \{l_0\}} \left( \sum_{j=0}^k A_j A_{k-j} \right) x^{k+1}+ \left (\sum_{j=0}^{l_0}A_j A_{l_0-j}+8^{l_0+1}\varepsilon\right )x^{l_0+1}\\
        &\geq \sum_{k= 0}^{\infty} \left( \sum_{j=0}^k A_j A_{k-j} \right) x^{k+1} +\varepsilon
        = xA(x)^2+\varepsilon
    \end{align}
    by Lemma \ref{lem-sequence} (b) and \eqref{eq47}. This implies that, for any fixed $x\in \left(\frac{1}{8},\frac{1}{4}\right)$, the following quadratic polynomial $x t^2 - t + (1+\varepsilon) \in \mathbb{R}[t]$ should have a real root, so we obtain the following inequality
    \begin{equation}
        1-4x(1+\varepsilon) \geq 0
    \end{equation}
    for arbitrary $x \in \left( \frac{1}{8}, \frac{1}{4}
    \right)$. Then, by taking the limit $x\rightarrow \frac{1}{4}$, we reach the following contradiction
    \begin{equation}
        -\varepsilon = 1-(1+\epsilon) \geq 0.
    \end{equation}
\end{proof}

\bibliographystyle{alpha}

\bibliography{JY25}

\begin{thebibliography}{Wan98}

\bibitem[BC07a]{BaCo07}
Teodor Banica and Beno\^it Collins.
\newblock Integration over compact quantum groups.
\newblock {\em Publ. Res. Inst. Math. Sci.}, 43(2):277--302, 2007.

\bibitem[BC07b]{BaCoper}
Teodor Banica and Beno\^it Collins.
\newblock Integration over quantum permutation groups.
\newblock {\em J. Funct. Anal.}, 242(2):641--657, 2007.

\bibitem[CC22]{CaCo22}
L.~Cadilhac and B.~Collins.
\newblock A metric characterization of freeness.
\newblock {\em J. Funct. Anal.}, 283(5):Paper No. 109562, 14, 2022.

\bibitem[Kes59]{Kes59}
Harry Kesten.
\newblock Symmetric random walks on groups.
\newblock {\em Trans. Amer. Math. Soc.}, 92:336--354, 1959.

\bibitem[KV00]{KuVa00}
Johan Kustermans and Stefaan Vaes.
\newblock Locally compact quantum groups.
\newblock {\em Ann. Sci. \'Ecole Norm. Sup. (4)}, 33(6):837--934, 2000.

\bibitem[KV03]{KuVa03}
Johan Kustermans and Stefaan Vaes.
\newblock Locally compact quantum groups in the von {N}eumann algebraic
  setting.
\newblock {\em Math. Scand.}, 92(1):68--92, 2003.

\bibitem[Wan95]{Wa95}
Shuzhou Wang.
\newblock Free products of compact quantum groups.
\newblock {\em Comm. Math. Phys.}, 167(3):671--692, 1995.

\bibitem[Wan98]{Wa98}
Shuzhou Wang.
\newblock Quantum symmetry groups of finite spaces.
\newblock {\em Comm. Math. Phys.}, 195(1):195--211, 1998.

\bibitem[Wor87]{Wo87b}
S.~L. Woronowicz.
\newblock Compact matrix pseudogroups.
\newblock {\em Comm. Math. Phys.}, 111(4):613--665, 1987.

\end{thebibliography}

\end{document}